\documentclass{amsart}

\usepackage{color}
\usepackage{tikz}
\newcommand{\gauss}[2]{\genfrac{[}{]}{0pt}{}{#1}{#2}}

\newtheorem{thm}{Theorem}[section]

\newtheorem{proposition}[thm]{Proposition}
\newtheorem{cor}[thm]{Corollary}

\newtheorem{dfn}[thm]{Definition}

\newtheoremstyle{definition}
{3pt} 
{3pt} 
{} 
{} 
{\bfseries} 
{.} 
{.5em} 
{} 

\title{The $A_2$ Rogers-Ramanujan identities revisited}

\date{\today}
\author{Sylvie Corteel}
\address{IRIF, CNRS et Universit\'e Paris Diderot,
France}
\email{corteel@irif.fr}
\author{Trevor Welsh}
\address{Department of Mathematics and Statistics,
University of Melbourne, Australia}
\email{twelsh1@unimelb.edu.au}
\begin{document}

\maketitle

{\centering\footnotesize To George Andrews for his $80^{th}$ birthday \par}

\vspace{.3cm}

\begin{abstract}
In this note we show how to use cylindric partitions to rederive the
$A_2$ Rogers-Ramanujan identities originally proven by Andrews,
Schilling and Warnaar.
\end{abstract}


\section{Introduction}

The Rogers-Ramanujan identities were first proved in 1894
by Rogers and rediscovered in the 1910s
by Ramanujan \cite{RR}.
They are
\begin{equation}
\sum_{n\ge 0} \frac{q^{n(n+i)}}{(q;q)_n}=\frac{1}{(q^{1+i};q^5)_\infty(q^{4-i};q^5)_\infty}.
\end{equation}
with $i=0,1$ and where $(a,q)_\infty=\prod_{i\ge 0}(1-aq^i)$ and $(a;q)_n=(a;q)_\infty/(aq^n;q)_\infty$.

There have been
many attempts to give combinatorial proofs of these identities
and the first one is due to Garsia and Milne \cite{GM}.
Unfortunately, it is not simple, and
no simple combinatorial proof is known.
Recently in \cite{C}, the first
author presented a new bijective approach to the proofs
of the Rogers-Ramanujan identities via the
Robinson-Schensted-Knuth correspondence as presented in \cite{Pak}. 
The bijection does not give the Rogers-Ramanujan identities but
the Rogers-Ramanujan identities divided by $(q;q)_\infty$; namely
\begin{equation}
\label{RR2}
\frac{1}{(q;q)_\infty}\sum_{n\ge 0} \frac{q^{n(n+1)}}{(q;q)_n}=\frac{1}{(q,q^2,q^2,q^3,q^3,q^4,q^5;q^5)_\infty}
\end{equation}
where $(a_1,\ldots , a_k;q)_\infty=\prod_{i=1}^k (a_i;q)_\infty$.
This proof uses the combinatorics of cylindric partitions \cite{GK}. We interpret both sides as the generating function of cylindric partitions of profile $(3,0)$ and the bijection is a polynomial algorithm in the size of the cylindric partition.
The idea to use cylindric partitions is due to Foda and  Welsh \cite{FW} in a more general setting:
the Andrews-Gordon identities  \cite{And}. For $k>0$ and  $0\le i\le k$,
these identities divided by $(q;q)_\infty$ are
$$
\frac{1}{(q;q)_\infty}\sum_{n_1,\ldots, n_k} \frac{q^{\sum_{j=1}^k n_j^2+\sum_{j=i}^k n_j}}
{(q)_{n_1-n_2}\ldots (q)_{n_{k-1}-n_k}(q)_{n_k}}=
\frac{(q^i,q^{2k+3-i},q^{2k+3};q^{2k+3})_\infty}{(q;q)_\infty^2}.
$$
Foda and Welsh  interpret the sum side as a generating function for
(what they call) decorated Bressoud paths and the product side is
interpreted as the generating function of cylindric partitions
of profile $(2k+1-i,i)$,and they provide a bijection between the two objects.
See \cite{FW} for more details.

In this note, we take the idea of applying cylindric partitions to
Rogers-Ramanujan type identities a step further,
by using them to give an alternative proof of the $A_2$ Rogers-Ramanujan
identities due to Andrews, Schilling and Warnaar \cite{ASW}.
\begin{thm} We have
\begin{equation*}
\begin{split}
\sum_{n_1=0}^\infty\sum_{n_2=0}^{2n_1} 
\frac{q^{n_1^2+n_2^2-n_1n_2+n_1+n_2}}{(q;q)_{n_1}}
	\gauss{2n_1}{n_2}
	&=\frac{1}{(q^2,q^3,q^3,q^4,q^4,q^5;q^7)_\infty},\\
\sum_{n_1=0}^\infty\sum_{n_2=0}^{2n_1}  \frac{q^{n_1^2+n_2^2-n_1n_2+n_2}}{(q;q)_{n_1}}
\gauss{2n_1}{n_2}
	&=\frac{1}{(q,q^2,q^3,q^4,q^5,q^6;q^7)_\infty},\\
\sum_{n_1=0}^\infty\sum_{n_2=0}^{2n_1+1}  \frac{q^{n_1^2+n_2^2-n_1n_2+n_1}}{(q;q)_{n_1}}
\gauss{2n_1+1}{n_2}
	&=\frac{1}{(q,q^2,q^3,q^4,q^5,q^6;q^7)_\infty},\\
\sum_{n_1=0}^\infty\sum_{n_2=0}^{2n_1+1} \frac{q^{n_1^2+n_2^2-n_1n_2+n_2}}{(q;q)_{n_1}}
\gauss{2n_1+1}{n_2}
	&=\frac{1} {(q,q^2,q^2,q^5,q^5,q^6;q^7)_\infty},\\
\sum_{n_1=0}^\infty\sum_{n_2=0}^{2n_1}
\frac{q^{n_1^2+n_2^2-n_1n_2}}{(q;q)_{n_1}}
\gauss{2n_1}{n_2}
	&=\frac{1}{(q,q,q^3,q^4,q^6,q^6;q^7)_\infty},\\
\end{split}
\end{equation*}
where the Gaussian polynomial $\gauss{n}{k}$ is defined by
$$
\gauss{n}{k}
	=\frac{(q;q)_n} {(q;q)_k(q;q)_{n-k}}.$$
\label{main}
\end{thm}
All but the fourth of these identities were obtained in
Theorem 5.2 of \cite{ASW}, while the fourth was conjectured in
Section 2.4 of \cite{FW2}.
Note that the second and third expressions are equal.

In this note, we prove the following theorem, giving the
generating functions $F_{c,n}(q)$ of cylindric partitions
indexed by compositions $c$ of 4 into 3 parts,
with largest entry at most $n$:

\begin{thm}
\begin{eqnarray*}
F_{(4,0,0),n}(q)&=&\sum_{n_1=0}^n\sum_{n_2=0}^{2n_1} \frac{q^{n_1^2+n_2^2-n_1n_2+n_1+n_2}}{(q;q)_{n-n_1}(q;q)_{n_1}}\left[\begin{array}{c}
2n_1\\ n_2\end{array}\right],\\
F_{(3,1,0),n}(q)&=&\sum_{n_1=0}^n\sum_{n_2=0}^{2n_1}  \frac{q^{n_1^2+n_2^2-n_1n_2+n_2}}{(q;q)_{n-n_1}(q;q)_{n_1}}\left[\begin{array}{c}
2n_1\\ n_2\end{array}\right],\\
F_{(3,0,1),n}(q)&=&\sum_{n_1=0}^n\sum_{n_2=0}^{2n_1}  \frac{q^{n_1^2+n_2^2-n_1n_2}}{(q;q)_{n-n_1}}
\frac{q^{n_1}}{(q;q)_{n_1}}\left[\begin{array}{c}2n_1\\ n_2\end{array}\right]\\
&&+\sum_{n_1=1}^n\sum_{n_2=0}^{2n_1-2}  \frac{q^{n_1^2+n_2^2-n_1n_2}}{(q;q)_{n-n_1}}
\frac{q^{2n_2}}{(q;q)_{n_1-1}}\left[\begin{array}{c}
2n_1-2\\ n_2\end{array}\right],\\
F_{(2,2,0),n}(q)&=&\sum_{n_1=0}^n\sum_{n_2=0}^{2n_1} \frac{q^{n_1^2+n_2^2-n_1n_2}}{(q;q)_{n-n_1}}\frac{q^{n_1}}{(q;q)_{n_1}}\left[\begin{array}{c}
2n_1\\ n_2\end{array}\right]\\
&&+\sum_{n_1=1}^n\sum_{n_2=0}^{2n_1-2} \frac{q^{n_1^2+n_2^2-n_1n_2}}{(q;q)_{n-n_1}}\frac{q^{n_2}(1+q^{n1+n2})}{(q;q)_{n_1-1}}\left[\begin{array}{c}
2n_1-2\\ n_2\end{array}\right],\\
F_{(2,1,1),n}(q)&=&\sum_{n_1=0}^n\sum_{n_2=0}^{2n_1} \frac{q^{n_1^2+n_2^2-n_1n_2}}{(q;q)_{n-n_1}(q;q)_{n_1}}\left[\begin{array}{c}
2n_1\\ n_2\end{array}\right].
\end{eqnarray*}
\label{new}
\end{thm}
This result gives a finite version of the sum side of the $A_2$ Rogers-Ramanujan identities.

In the $n\to\infty$ limit,
we recover the sum side of the identities of
Theorem \ref{main} divided by $(q;q)_\infty$.
We also explain how to get the product side. It is 
a corollary of a result of Borodin \cite{Bor}.
In Section \ref{prod} below,
we start by defining cylindric partitions and then obtain the product
sides of particular cylindric partitions.
These yield the right hand sides of the expressions in Theorem \ref{main}.
The sum expressions on the left hand sides are computed in Section \ref{sum}.

\medskip
 
\noindent{\bf Acknowledgments.}
SC was in residence at MSRI  (NSF grant DMS-1440140) and was
visiting the Mathematics
department at UC Berkeley during the completion of this work.
TW acknowledges partial support from the Australian Research Council.
The authors wish to thank Omar Foda for his interest
in this work and useful discussions.
The authors also wish to thank the anonymous referee
for her excellent suggestions and careful reading. 

\section{Cylindric partitions and the product side}
\label{prod}

Cylindric partitions were introduced by Gessel and Krattenthaler \cite{GK} and appeared naturally in different contexts
\cite{Bor,BCC,CSV,G,FW,Tin}. Let  $\ell$ and $k$ be two positive integers. 
In this note, we choose to index cylindric partitions by compositions of $\ell$ into $k$ non negative parts.
\begin{dfn}
Given a composition $c=(c_1,\ldots ,c_k)$,  a cylindric partition  of profile $c$ 
is a sequence of $k$ partitions
$\Lambda=(\lambda^{(1)},\ldots \lambda^{(k)})$ such that~:
\begin{itemize}
\item $
\lambda^{(i)}_j\ge \lambda^{(i+1)}_{j+c_{i+1}},
$
\item $\lambda^{(k)}_j\ge \lambda^{(1)}_{j+c_1}$.
\end{itemize}
for all $i$ and $j$.
\end{dfn}
For example, the sequence $\Lambda=((3,2,1,1),(4,3,3,1),(4,1,1))$
is a cylindric partition of profile $(2,2,0)$. 
One can check that for all $j$, $\lambda^{(1)}_j\ge \lambda^{(2)}_{j+2}$, $\lambda^{(2)}_j\ge \lambda^{(3)}_{j}$
and $\lambda^{(3)}_j\ge \lambda^{(1)}_{j+2}$ for all $j$.
Note that this definition implies that cylindric partitions of profile $(c_1,\ldots ,c_k)$ are in bijection with cylindric
partitions of profile $(c_k,c_1,\ldots ,c_{k-1})$.

Our goal is to compute  generating functions of cylindric partitions of a given profile $c$ according to
two statistics. Given a $\Lambda=(\lambda^{(1)},\ldots , \lambda^{(k)})$, let
\begin{itemize}
\item $|\Lambda|=\sum_{i=1}^k \sum_{j\ge 1} \lambda^{(i)}_j$, the sum of the entries of the cylindric
plane partition, and
\item $\max(\Lambda)=\max(\lambda^{(1)}_1,\ldots \lambda^{(k)}_1)$, the largest entry of the cylindric
plane partition.
\end{itemize}
Going back to our example, we have $|\Lambda|=24$, and $\max(\Lambda)=4$.

Let ${\mathcal C_c}$ be the set of cylindric partitions of profile $c$ and let  ${\mathcal C_{c,n}}$ be the set of cylindric partitions of profile $c$ and such that the largest entry is at most $n$.
We are interested in the following generating functions.
\begin{align}
\label{Def:Fcq}
F_c(q)&=\sum_{\Lambda\in \mathcal C_c} q^{|\Lambda|},\\ 
\label{Def:Fcyq}
F_c(y,q)&=\sum_{\Lambda\in \mathcal C_c} q^{|\Lambda|}y^{\max(\Lambda)},\\
\label{Def:Fcnq}
F_{c,n}(q)&=\sum_{\Lambda\in \mathcal C_{c,n}} q^{|\Lambda|}.
\end{align}

A surprising and beautiful result is that for any $c$,
the generating function $F_c(q)$ can be
written as a product.
Namely,
with $t=k+\ell$,
\begin{thm}\cite{Bor}
The generating function $F_c(q)$ is equal to
\begin{equation}
\frac{1}{(q^t;q^t)}\cdot
\prod_{i=1}^k\prod_{j=i+1}^k\prod_{m=1}^{c_i}\frac{1}{(q^{m+d_{i+1,j}+j-i};q^t)_\infty}
\cdot
\prod_{i=2}^k\prod_{j=2}^{i-1}\prod_{m=1}^{c_{i}}\frac{1}{(q^{t-(m+d_{j,i-1}+i-j)};q^t)_\infty}.
\end{equation}
where $d_{i,j}=c_i+c_{i+1}+\ldots +c_j$.
\end{thm}
The original result is written is a  different but equivalent form.

For what follows, we restrict attention to the
case $\ell=4$ and $k=3$.
As cylindric partitions of profile $(c_1,\ldots ,c_k)$ are in bijection with 
partitions of profile $(c_k,c_1,\ldots ,c_{k-1})$, we need
only  compute the generating functions for
the compositions $(4,0,0)$,  $(3,1,0)$, $(3,0,1)$, $(2,2,0)$, and $(2,1,1)$. 
We now apply the previous theorem:
\begin{cor}
\begin{equation*}
\begin{split}
F_{(4,0,0)}(q)&=\frac{1}{(q;q)_\infty(q^2,q^3,q^3,q^4,q^4,q^5;q^7)_\infty},\\
F_{(3,1,0)}(q)&=\frac{1}{(q;q)_\infty(q,q^2,q^3,q^4,q^5,q^6;q^7)_\infty},\\
F_{(3,0,1)}(q)&=\frac{1}{(q;q)_\infty(q,q^2,q^3,q^4,q^5,q^6;q^7)_\infty},\\
F_{(2,2,0)}(q)&=\frac{1}{(q;q)_\infty(q,q^2,q^2,q^5,q^5,q^6;q^7)_\infty},\\
F_{(2,1,1)}(q)&=\frac{1}{(q;q)_\infty(q,q,q^3,q^4,q^6,q^6;q^7)_\infty}.\\
\end{split}
\end{equation*}
\label{43}
\end{cor}

\noindent Note that these five products are precisely those in
Theorem \ref{main} divided by $(q;q)_\infty$.

\section{The sum side}
\label{sum}

We first prove a general functional equation for $F_c(y,q)$ for any profile $c$.
Suppose that $k>1$ and $c=(c_1,\ldots, c_k)$. 
Let $I_c$ be the subset of $\{1,\ldots ,k\}$ such that $i\in I_c$ if and only if $c_i>0$.
For example if $c=(2,2,0)$ then $I_c=\{1,2\}$.
Given a subset $J$ of $I_c$, we define the composition $c{(J)}=(c_1(J),\ldots, c_k(J))$ by
$$
c_i{(J)}=\left\{\begin{array}{ll}
c_i-1 & {\rm if} \ {i\in J} \ {\rm and} \ {(i-1)\not\in J,}\\
c_i+1 & {\rm if} \ {i\notin J} \ {\rm and} \ {(i-1)\in J,}\\
c_i & {\rm otherwise.}\end{array}\right.
$$ 
Here we set $c_0=c_k$.
\begin{proposition} For any composition $c=(c_1,\ldots ,c_k)$,
\begin{equation}\label{Eq:INEX}
 F_c(y,q)=\sum_{\emptyset\subset J\subseteq I_c} (-1)^{|J|-1}  \frac{F_{c{(J)}}(yq^{|J|},q)}{1-yq^{|J|}}.
\end{equation}
with the conditions $F_c(0,q)=1$ and  $F_c(y,0)=1$.
\label{incex}
\end{proposition}

\begin{proof}
The proof make use of an inclusion-exclusion argument.

First, for fixed $J$ such that $\emptyset \subset J\subseteq I_c$,
we require the generating function
of cylindric partitions $\Lambda$ of profile $c$ such that
	$\lambda_1^{(j)}=\max(\Lambda)$ for all $j\in J$.

Let $M=(\mu^{(1)},\ldots ,\mu^{(k)})$ be a cylindric partition of
profile $c{(J)}$,
and set $n=\max(M)$.
Then, for a fixed integer $m\ge 0$,
create a cylindric partition
$\Lambda=(\lambda^{(1)},\ldots ,\lambda^{(k)})$
using the following recipe: 
\begin{equation*}
\lambda^{(j)}=
\begin{cases}
(m+n,\mu^{(j)}_1,\mu^{(j)}_2,\ldots )
	& \text{if $j\in J$,}\\
\mu^{(j)}
	& \text{if $j\notin J$.}
\end{cases}
\end{equation*}
It is easily checked that $\Lambda$ is a cylindric partition of profile
$c$ and that $\max(\Lambda)=m+n$.
Moreover, $\lambda_1^{(j)}=\max(\Lambda)$ for all $j\in J$.
The generating function for all cylindric partitions $\Lambda$
obtained from $M$ in this way is
\begin{equation}
\sum_{m=0}^\infty y^{m+n}q^{|J|(m+n)}q^{|M|}
= y^nq^{|J|n+|M|}\sum_{m=0}^\infty (yq^{|J|})^m
	= \frac{y^nq^{|J|n+|M|}}{1-yq^{|J|}}.
\end{equation}
Then the generating function for all cylindric partitions $\Lambda$
obtained in this way from any cylindric partition $M$ of profile
$c(J)$ is 
\begin{equation}\label{Eq:GFJ}
\sum_{M\in {\mathcal C}_{c{(J)}}}
	\frac{y^{\max(M)}q^{|J|\max(M)+|M|}}{1-yq^{|J|}}
= \frac{F_{c{(J)}}(yq^{|J|},q)}{1-yq^{|J|}},
\end{equation}
making use of the definition \eqref{Def:Fcyq}.

Let $\Lambda=(\lambda^{(1)},\ldots ,\lambda^{(k)})$ be an arbitrary
cylindric partition of profile $c$, and let $p=\max(\Lambda)$.
Because $\lambda_1^{(i-1)}\ge \lambda_1^{(i)}$ whenever $i\notin I_c$,
it must be the case that $p=\lambda^{(j)}_1$ for some $j\in I_c$.
Then, if $J\neq\emptyset$ is such that $p=\lambda^{(j)}_1$ for each $j\in J$
(this $J$ might not be unique), we see that
$\Lambda$ is one of the cylindric partitions enumerated by \eqref{Eq:GFJ}.
However, because $\Lambda$ can arise from various different $J$,
the generating function for cylindric partitions of profile $c$
is obtained via the inclusion-exclusion process.
This immediately gives \eqref{Eq:INEX}.
\end{proof}

Now, for each composition $c$, define
\begin{equation}\label{Def:G}
G_{c}(y,q)=
(yq;q)_\infty\,
F_c(y,q)
.
\end{equation}
In terms of this, the previous result translates to
\begin{equation}
G_c(y,q)=\sum_{\emptyset\subset J\subseteq I} (-1)^{|J|-1}  (yq;q)_{|J|-1}{G_{c{(J)}}(yq^{|J|},q)}
\label{Gb}
\end{equation}
with $G_c(0,q)=G_c(y,0)=1$.

\begin{thm}\label{Thm:G}
\begin{equation*}
\begin{split}
G_{(4,0,0)}(y,q)
&=\sum_{n_1=0}^\infty\sum_{n_2=0}^{2n_1}
y^{n_1}\frac{q^{n_1^2+n_2^2-n_1n_2+n_1+n_2}}{(q;q)_{n_1}}
\gauss{2n_1}{n_2},
\\
G_{(3,1,0)}(y,q)
&=\sum_{n_1=0}^\infty\sum_{n_2=0}^{2n_1}
y^{n_1}\frac{q^{n_1^2+n_2^2-n_1n_2+n_2}}{(q;q)_{n_1}}
\gauss{2n_1}{n_2},
\\
G_{(3,0,1)}(y,q)
&=\sum_{n_1=0}^\infty\sum_{n_2=0}^{2n_1} y^{n_1}{q^{n_1^2+n_2^2-n_1n_2}}
\frac{q^{n_1}}{(q;q)_{n_1}}
\gauss{2n_1}{n_2}
\\
&\quad+\sum_{n_1=1}^\infty\sum_{n_2=0}^{2n_1-2}
y^{n_1}{q^{n_1^2+n_2^2-n_1n_2}}\frac{q^{2n_2}}{(q;q)_{n_1-1}}
\gauss{2n_1-2}{n_2},
\\
G_{(2,2,0)}(y,q)
&=\sum_{n_1=0}^\infty\sum_{n_2=0}^{2n_1} y^{n_1}{q^{n_1^2+n_2^2-n_1n_2}}
\frac{q^{n_1}}{(q;q)_{n_1}}
\gauss{2n_1}{n_2}
\\
&\quad+\sum_{n_1=1}^\infty\sum_{n_2=0}^{2n_1-2}
y^{n_1}{q^{n_1^2+n_2^2-n_1n_2}}\frac{q^{n_2}(1+q^{n1+n2})}{(q;q)_{n_1-1}}
\gauss{2n_1-2}{n_2},
\\
G_{(2,1,1)}(y,q)
&=\sum_{n_1=0}^\infty\sum_{n_2=0}^{2n_1}
y^{n_1}\frac{q^{n_1^2+n_2^2-n_1n_2}}{(q;q)_{n_1}}
\gauss{2n_1}{n_2}.
\end{split}
\end{equation*}
\end{thm}
\begin{proof}
In this proof, we abbreviate $G_c(y,q)$ to $G_c(y)$ for convenience.
Applying the form \eqref{Gb} of Propostion \ref{incex}
to the case $\ell=4$ and $k=3$ yields
\begin{equation*}
\begin{split}
G_{(4,0,0)}(y)&=G_{(3,1,0)}(yq),\\
G_{(3,1,0)}(y)&=G_{(3,0,1)}(yq)+G_{(2,2,0)}(yq)-(1-yq)G_{(2,1,1)}(yq^2),\\
G_{(3,0,1)}(y)&=G_{(4,0,0)}(yq)+G_{(2,1,1)}(yq)-(1-yq)G_{(3,1,0)}(yq^2),\\
G_{(2,2,0)}(y)&=G_{(3,0,1)}(yq)+G_{(2,1,1)}(yq)-(1-yq)G_{(2,1,1)}(yq^2),\\
G_{(2,1,1)}(y)&=G_{(2,1,1)}(yq)+G_{(2,2,0)}(yq)+G_{(3,1,0)}(yq)\\
&\quad-(1-yq)(G_{(2,2,0)}(yq^2)+G_{(2,1,1)}(yq^2)+G_{(3,0,1)}(yq^2))\\
&\quad+(1-yq)(1-yq^2)G_{(2,1,1)}(yq^3).\\
\end{split}
\end{equation*}
By manipulating these equations, we obtain
\begin{equation}\label{Eq:Fun}
\begin{split}
G_{(4,0,0)}(y)&=G_{(3,1,0)}(yq),\\
G_{(3,1,0)}(y)&=G_{(2,2,0)}(yq)+yq^2 G_{(3,1,0)}(yq^3)+yq G_{(2,1,1)}(yq^2),\\
G_{(3,0,1)}(y)&=G_{(2,1,1)}(yq)+yq G_{(3,1,0)}(yq^2),\\
G_{(2,2,0)}(y)&=G_{(2,1,1)}(yq)+yqG_{(2,1,1)}(yq^2)+yq^2 G_{(3,1,0)}(yq^3),\\
G_{(2,1,1)}(y)&=G_{(2,1,1)}(yq)+yq G_{(2,2,0)}(yq)+yq G_{(2,2,0)}(yq^2)\\
&\quad+yq^3 G_{(3,1,0)}(yq^4)+yq^2 G_{(2,1,1)}(yq^3).
\end{split}
\end{equation}
We claim that this system of equations \eqref{Eq:Fun} together with
the boundary conditions $G_c(0,q)=G_c(y,0)=1$ for each composition $c$,
is uniquely solved by the expressions stated in the theorem.
This is proved using an induction argument involving all five expressions.

We use induction on the exponents of $y$.
For each composition $c$, let $g_c(n)$ denote the coefficient
of $y^n$ in the solution $G_c(y)$ of \eqref{Eq:Fun}.
The boundary conditions $G_c(0,q)=1$ imply that each $g_c(0)=1$.
This holds for the expressions of the theorem.
So now, for $n>0$, assume that $g_c(n_1)$ agrees with the coefficient
of $y^{n_1}$ in the statement of the theorem for each $n_1<n$ and
each composition $c$.
We must check that $g_c(n)$, as determined by the expressions \eqref{Eq:Fun},
is equal to the coefficient of $y^n$ in the statement of the theorem for
each $c$.

The fifth expression in \eqref{Eq:Fun} implies that
\begin{equation*}
\begin{split}
(1-q^{n})g_{(2,1,1)}(n)&=
(q^{n}+q^{2n-1}) g_{(2,2,0)}(n-1)\\
&\qquad+ q^{4n-1}g_{(3,1,0)}(n-1)+q^{3n-1} g_{(2,1,1)}(n-1).
\end{split}
\end{equation*}
Using the expressions for $g_c(n)$ implied by the induction
hypothsis then yields
\begin{equation*}
g_{(2,1,1)}(n)=\sum_{n_2=0}^{2n} \frac{q^{n^2+n_2^2-nn_2}}{(q;q)_{n}}\left[\begin{array}{c}
2n\\ n_2\end{array}\right].
\end{equation*}
This expression, along with the other expressions for
$g_c(n)$ implied by \eqref{Eq:Fun}, enables us to compute,
in turn, $g_{(2,2,0)}(n)$, $g_{(3,0,1)}(n)$,
$g_{(3,1,0)}(n)$ and finally $g_{(4,0,0)}(n)$.
Because the expressions that result agree with the corresponding
coefficients of $y^n$ in the statement of the theorem,
the induction argument is complete.
\end{proof}

\noindent {\bf Proof of Theorem  \ref{new}.}
Comparing \eqref{Def:Fcyq} and \eqref{Def:Fcnq} leads to
\begin{equation*}
\begin{split}
F_{c,n}(q)&=[y^0] F_c(y,q)+ [y^1] F_c(y,q) + \cdots + [y^n] F_c(y,q)\\
&=[y^n] \frac{F_c(y,q)}{1-y}
= [y^n] \frac{G_c(y,q)}{(y;q)_n}
\end{split}
\end{equation*}
having used \eqref{Def:G}.
By the $q$-binomial theorem \cite{AndBook}, we have
\begin{equation*}
\frac{1}{(y,q)_n}=\sum_{i=0}^n \frac{y^i}{(q;q)_i},
\end{equation*}
and therefore it follows that
\begin{equation*}
F_{c,n}(q)
=
\sum_{n_1=0}^n \frac{1}{(q;q)_{n-n_1}}[y^{n_1}]G_c(y,q).
\end{equation*}
Applying this to the expressions of Theorem \ref{Thm:G}
then yields those of Theorem \ref{new}.
\qed

\noindent {\bf Proof of Theorem  \ref{main}.}
Let us now comment on how to get the end of the proof of Theorem \ref{main}.
The right hand side of the five identities correspond to Corollary \ref{43}. To get the sum side;
we first  let $n$ tend to infinity in Theorem \ref{new}
and multiply by $(q;q)_\infty$.
We get directly the left hand side of the first, second and fifth identities. 
Let us show how to get the left hand side third identity.
Theorem  \ref{new} states that the generating function of the cylindric partitions of profile $(3,0,1)$ and largest entry at most $n$  is:
\begin{equation}
F_{(3,0,1),n}(q)=\sum_{n_1,n_2} \frac{q^{n_1^2+n_2^2-n_1n_2}}{(q;q)_{n-n_1}(q;q)_{n_1}}
\left(q^{n_1}\left[\begin{array}{c}
2n_1\\ n_2\end{array}\right]+q^{2n_2}(1-q^{n_1})\left[\begin{array}{c}
2n_1-2\\ n_2\end{array}\right]\right)\
\label{(3,0,1)}
\end{equation}
We let $n\rightarrow \infty$  and multiply by $(q;q)_\infty$ in \eqref{(3,0,1)},
we get
\begin{equation*}
\sum_{n_1,n_2} \frac{q^{n_1^2+n_2^2-n_1n_2+n_1}}{(q;q)_{n_1}}
\left[\begin{array}{c}
2n_1\\ n_2\end{array}\right]+\sum_{n_1,n_2}\frac{q^{n_1^2+n_2^2-n_1n_2+2n_2}}{(q;q)_{n_1-1}}
\left[\begin{array}{c}
2n_1-2\\ n_2\end{array}\right]\\
\end{equation*}
After the change $(n_1,n_2)\rightarrow(n_1+1,n_2-1)$ in the second sum, we get
\begin{eqnarray*}
&=& \sum_{n_1,n_2} \frac{q^{n_1^2+n_2^2-n_1n_2+n_1}}{(q;q)_{n_1}}
\left[\begin{array}{c}
2n_1\\ n_2\end{array}\right]\\&&+\sum_{n_1,n_2}\frac{q^{(n_1-1)^2+(n_2+1)^2-(n_1-1)(n_2+1)+(n_1-1)+(2(n_1-1)-(n_2+1)+1)}}{(q;q)_{n_1-1}}
\left[\begin{array}{c}
2n_1-2\\ n_2\end{array}\right]\\
&=& \sum_{n_1,n_2} \frac{q^{n_1^2+n_2^2-n_1n_2+n_1}}{(q;q)_{n_1}}
\left[\begin{array}{c}
2n_1\\ n_2\end{array}\right]+\sum_{n_1,n_2}\frac{q^{n_1^2+n_2^2-n_1n_2+n_1+(2n_1-n_2+1)}}{(q;q)_{n_1}}
\left[\begin{array}{c}
2n_1\\ n_2-1\end{array}\right]\\
&=&\sum_{n_1,n_2} \frac{q^{n_1^2+n_2^2-n_1n_2+n_1}}{(q;q)_{n_1}}
\left[\begin{array}{c}
2n_1+1\\ n_2\end{array}\right],
\end{eqnarray*}
using the well known recurrence relation
$$
\left[\begin{array}{c}
n\\ k\end{array}\right]=\left[\begin{array}{c}
n-1\\ k\end{array}\right]+q^{n-k}\left[\begin{array}{c}
n-1\\ k-1\end{array}\right].
$$
This is the left hand side of the third identity of Theorem \ref{main}.

We leave the computation for  the left hand side of the fourth identity to the reader.
One needs to show that
$\displaystyle{
\sum_{n_1,n_2} \frac{q^{n_1^2+n_2^2-n_1n_2+n_2}}{(q;q)_{n_1}}
\left[\begin{array}{c}
2n_1+1\\ n_2\end{array}\right]}$
equals
$${\sum_{n_1,n_2}\frac{q^{n_1^2+n_2^2-n_1n_2}}{(q;q)_{n_1}}
\left(q^{n_1}\left[\begin{array}{c}
2n_1\\ n_2\end{array}\right]+q^{n_2}(1+q^{n1+n2})(1-q^{n_1})\left[\begin{array}{c}
2n_1-2\\ n_2\end{array}\right]\right).}$$
\qed


\begin{thebibliography}{99}
\bibitem{AndBook} G.E. Andrews, The Theory of Partitions. Cambridge University Press. ISBN 0-521-63766-X, 1976.
\bibitem{And} G.E. Andrews, On the General Rogers-Ramanujan Theorem. Providence, RI: Amer. Math. Soc., 1974.
\bibitem{ASW} G.E. Andrews, A. Schilling, S.O. Warnaar, An $A_2$ Bailey lemma and Rogers-Ramanujan-type identities,
J. Amer. Math. Soc.,  12 (3), 677-702, 1999.
\bibitem{Bor} A. Borodin,
Periodic Schur process and cylindric partitions. Duke Math. J. 140 (2007), no. 3, 391–468. 
\bibitem{BCC} J. Bouttier, G. Chapuy and S. Corteel, 
From Aztec diamonds to pyramids: steep tilings,  Trans. Amer. Math. Soc.  369  (2017),  no. 8, 5921--5959.
\bibitem{C} S. Corteel, Rogers-Ramanujan identities and
 the Robinson-Schensted-Knuth Correspondence,  Proc. Amer. Math. Soc.  145  (2017),  no. 5, 2011--2022.
\bibitem{CSV} S. Corteel, C. Savelief and M. Vuletic, Plane overpartitions and cylindric partitions, Jour. of Comb. Theory A, Vol.118, Issue 4, (2011), 1239-1269.  
\bibitem{FW2} B. Feigin, O. Foda, and T. Welsh, 
Andrews-Gordon type identities from combinations of Virasoro characters. Ramanujan J. 17 (2008), no. 1, 33-52. 
\bibitem{G} T. Gerber, Crystal isomorphisms in Fock spaces and Schensted correspondence in affine type A. 
Algebras and Representation Theory 18 (2015), 1009-1046.
\bibitem{GK} Ira Gessel, and Christian Krattenthaler, Cylindric partitions. 
Trans. Amer. Math. Soc. 349 (1997), no. 2, 429-479.
\bibitem{GM} A. Garsia, and S.C. Milne,  A Rogers-Ramanujan Bijection. J. Combin. Th. Ser. A 31, 289-339, (1981).
\bibitem{FW} O. Foda and T. Welsh, Cylindric partitions, $W_r$ characters and the Andrews-Gordon-Bressoud identities, Journal of Physics. A, Mathematical and theoretical, 49(16), [164004]. arXiv:1510.02213. 
\bibitem{Pak} I. Pak, Partition bijections, a survey. Ramanujan J. 12 (2006), no. 1, 5-75.
\bibitem{RR} L.J. Rogers and Srinivasa Ramanujan, Proof of certain identities in combinatorial analysis, Camb. Phil. Soc. Proc., Vol 19, 1919, p. 211-216.
\bibitem{Tin} P. Tingley, Three combinatorial models for $\widehat{sl}_n$ crystals, with applications to 
cylindric partitions. Int. Math. Res. Not. IMRN 2008, no. 2, Art. ID rnm143, 40 pp.
\end{thebibliography}
\end{document}